	\DeclareSymbolFont{AMSb}{U}{msb}{m}{n} 			
	\documentclass[12pt,oneside,noamsfonts,a4paper]{amsart}

	\usepackage[utf8]{inputenx}  
	\usepackage[T1]{fontenc}     

	\usepackage[USenglish]{babel}

	\usepackage[charter,expert,sfscaled]{mathdesign}  
		\renewcommand{\in}{\smallin}
		\renewcommand{\notin}{\notsmallin}
		\renewcommand{\setminus}{\smallsetminus}

	\usepackage[scaled=.96,sups]{XCharter}            
	\usepackage[kerning]{microtype}  
	\DeclareMicrotypeAlias{XCharter-TLF}{bch} 

	\usepackage{mathtools}       
	\usepackage{mathscinet}			
	\usepackage{xcolor}
	\usepackage{mleftright}                       

	{\newlength\figurewidth}
	\newcommand\fdash{\settowidth\figurewidth{9}\raisebox{0.6ex}{\makebox[\figurewidth]{\hrulefill}}}

		\newtheorem{theorem}{Theorem}
		\newtheorem{proposition}[theorem]{Proposition}
		\newtheorem{lemma}[theorem]{Lemma}
		
		\newtheorem*{claim*}{Claim}
		\newtheorem{corollary}[theorem]{Corollary}
		\theoremstyle{definition}
		\newtheorem{definition}[theorem]{Definition}
		
		\newtheorem{question}[theorem]{Question}
		\newtheorem{fact}[theorem]{Fact}
		\newtheorem{observation}[theorem]{Observation}
		\theoremstyle{remark}
		\newtheorem{remark}[theorem]{Remark}

		\newenvironment{claimproof}{}{}  
		\newcommand{\claimdone}{\hfill$\blacksquare$\par}

		\newcommand{\pw}[1]{\mathcal{P}\left(#1\right)}  
		\newcommand{\fin}{\textup{fin}}          
		\newcommand{\os}{\mleft\{ \,}             
		\newcommand{\cs}{\, \mright\}}             
		\newcommand{\set}[1]{\os#1\cs}                
		\newcommand{\la}{\mleft\langle\,}           
		\newcommand{\ra}{\, \mright\rangle}          
		\newcommand{\sq}[1]{\la#1\ra}              
		\newcommand{\card}[1]{\mleft| #1 \mright|}    

		\DeclareMathOperator{\cof}{cof}
		
		\DeclareMathOperator{\dom}{dom}
		\newcommand{\cf}{\operatorname{\mathrm{cf}}}

		\newcommand{\conc}{\mathbin{\!{}^\smallfrown\!}}           
		\renewcommand{\restriction}{\mathbin{\!\upharpoonright}}   

		\newcommand{\comb}[1]{\operatorname{\mathrm{comb}} \mleft(#1 \mright)}

		\renewcommand{\mid}{\shortmid}             

	\title{Free sequences in $\pw{\omega}/\fin$}

		\author{David Chodounsk\'y}
		\address{Institute of Mathematics of the Czech Academy of Sciences,
		\v{Z}itn\'{a}~25, Praha~1, Czech Republic}
		\email{chodounsky@math.cas.cz}

		\author{Vera Fischer}
		\address{Kurt G\"{o}del Research Center, University of Vienna,
		Währinger straße~25, 1090 Wien, Austria}
		\email{vera.fischer@univie.ac.at}

		\author{Jan Greb\'{\i}k}
		\address{Institute of Mathematics of the Czech Academy of Sciences,
		\v{Z}itn\'{a}~25, Praha~1, Czech Republic}
		\email{grebik@math.cas.cz}

	\subjclass[2010]{03E17, 03E35, 06E05}
	\keywords{maximal free sequence, dense independent system, party forcing}

	\thanks{
		The first author was supported by the GA~CR project 17\fdash33849L and RVO:\ 67985840.
		The second author would like to thank FWF for the generous support through grant number Y1012{\fdash}N35.
		The third author was supported by the GA~CR project 15\fdash34700L and RVO:\ 67985840.
	}

\begin{document}

\begin{abstract}
	We investigate maximal free sequences in the Boolean algebra $\pw{\omega}/\fin$,
	as defined by D.\ Monk in~\cite{monk}.
	We provide some information on the general structure of these objects
	and we are particularly interested in the minimal cardinality
	of a free sequence, a~cardinal characteristic of 
	the continuum denoted~$\mathfrak{f}$.
	Answering a~question of Monk, we demonstrate the consistency of
	$\omega_1 = \mathfrak{i} = \mathfrak{f} < \mathfrak{u} = \omega_2$.
	In fact, this consistency is demonstrated in the model of S.\ Shelah
	for~$\mathfrak{i}< \mathfrak{u}$~\cite{shelah-i<u}.
	Our paper provides a streamlined and mostly self contained
	presentation of this construction.
\end{abstract}

\maketitle

\section{Introduction}\label{sec:intro}

\noindent
The paper uses the following convention: For an element $a$ of a Boolean algebra~$B$
we denote $a^0$ the complement of $a$, occasionally we also use $a^1$ to denote $a$.
This convention is used even for subsets of $\omega$ which are considered as
elements of the Boolean algebra $\pw{\omega}$.

Free sequences in Boolean algebras were explicitly defined
by Donald Monk in~\cite{monk}.

\begin{definition}
	Sequence $A = \la a_\alpha \mid \alpha \in \gamma \ra$ of elements of a Boolean algebra
	of ordinal length $\gamma$ is a \emph{free sequence} if the family
	$\set{ a^1_\alpha \mid \alpha < \beta } \cup \set{ a^0_\alpha \mid \beta \leq \alpha < \gamma }$
	is centered for each $\beta \leq \gamma$.
\end{definition}

The concept of free sequences comes from an analogous notion in topological spaces.
A sequence of points $\la x_\alpha \mid \alpha < \gamma \ra$ in a topological space
is a free sequence if the topological closure of $\la x_\alpha \mid \alpha < \beta \ra$
is disjoint from the topological closure of $\la  x_\alpha \mid \beta \leq \alpha < \gamma \ra$
for each $\beta \leq \gamma$.
These objects were first consider by A.\ Arhangel'ski\u{\i} in~\cite{arhangelski}
who introduced this concept in order to solve a famous problem of
Alexandroff and Urysohn about the bound on the cardinality of first countable compact spaces.
In the topological context, the most important consideration seems to be
the maximal possible size of a free sequence, this gives rise to a cardinal invariant
of a topological space closely related to the tightness of the space,
see e.g.~\cite{bella,stevo-SandL}, and existence of $\piup$-bases with additional
combinatorial properties~\cite{gorelic}.

S.\ Todor\v{c}evi\'{c} defined an algebraic version of the topological notion of a free sequence
in~\cite{stevo-FS} and demonstrated that the algebraic formulation is often more convenient
than the original topological concept (see also~\cite{stevo-baire}).
For compact zero-dimensional topological spaces the algebraic definition of Todor\v{c}evi\'{c}
coincides via the Stone duality with the notion of a free sequence in a Boolean algebra
as defined by Monk. Nevertheless, the notion of a free sequence in a Boolean algebra
is not precisely dual to the notion of a free sequence of points in a topological space,
see the discussion in~\cite{monk}.

A free sequence $\la a_\alpha \mid \alpha \in \gamma \ra$
is \emph{maximal} if it is maximal with respect to end-extension,
i.e.\ there exist no $a_\gamma$ such that
$\la a_\alpha \mid \alpha \in \gamma \ra \conc \la a_\gamma \ra$
is also a free sequence.
Monk was primarily interested in the spectrum of possible cardinalities of maximal
free sequences in Boolean algebras. Most notably, for a Boolean algebra~$B$
he defined $\mathfrak f(B)$ to be the cardinal
$\min \set{ \card{A} \mid A \text{ is a maximal free sequence in } B }$.
Monk investigated the relation of this cardinal
with other cardinal characteristics of Boolean algebras.
Let us remark at this point that the relation of the cardinal spectrum
of possible cardinalities of maximal free sequences of a given Boolean algebra
with the ordinal spectrum of the actual ordinal lengths of maximal free sequences is quite unclear.
Even the question whether $\mathfrak f(B)$ is realized
by a maximal free sequence of ordinal length exactly $\mathfrak f(B)$
is in general quite non-trivial.

One of the main problems stated in~\cite{monk} was the relation of $\mathfrak f(B)$
and the ultrafilter number $\mathfrak u(B)$; the minimal size of an ultrafilter base in~$B$.
One of the instances of this problem was solved by K.\ Selker~\cite{selker}
who used forcing to demonstrate that the existence of a Boolean algebra $B$
such that $\omega = \mathfrak f(B) < \mathfrak u(B) = \omega_1$
is consistent with~\textsf{ZFC+CH}.

The present paper is solely interested in free sequences in the Boolean algebra
$\pw{\omega}/\fin$.
We make several observations on free sequences and the relation of the free sequence number
with other cardinal characteristics of the continuum.
Most notably, we prove that the free sequence number is strictly smaller than
the ultrafilter number $\mathfrak u$ in the model for $\mathfrak i < \mathfrak u$ of
Shelah~\cite{shelah-i<u}. As the paper of Shelah is considered to be somewhat cryptic,
we opted for providing a streamlined, complete and mostly self contained presentation
of the forcing construction from~\cite{shelah-i<u}.
All the core ingredients of this construction are originally due to Shelah.
Our contribution, apart of the presentation,
is the argument concerning free sequences and the free sequence number $\mathfrak f$.
Reader interested only in Shelah's construction may
skip Section~\ref{sec:basic} and other parts of this paper
which are concerned with free sequences.

\section{Basic considerations}\label{sec:basic}

We will start with exploring basic facts about possible incarnations
of maximal free sequences in $\pw{\omega}/\fin$.\footnote{
	We will not formally distinguish between the elements of the Boolean algebra
	$\pw{\omega}/\fin$ and their representatives in $\pw{\omega}$.
	We write $a \subset^* b$ when $b \setminus a$ is finite.
}
We define the \emph{free sequence number} $\mathfrak f$ to be the minimal cardinality
of a maximal free sequence in $\pw{\omega}/\fin$,
i.e.\ $\mathfrak f = \mathfrak f (\pw{\omega}/\fin)$.
For a given free sequence $A = \la a_\alpha \in \pw{\omega} \mid \alpha < \gamma \ra$
we denote the set of admissible intersections as
\[ \comb{A} = \set{ \bigcap_{\alpha \in \Gamma} a_\alpha \cap \bigcap_{\alpha \in \Delta} a^0_\alpha
		\mid \Gamma, \Delta \in {[\gamma]}^{< \omega}, \Gamma < \Delta  }. \]
We will also consider the filter generated by a free sequence,
this is just the filter the free sequence generates as a centered subset of $\pw{\omega}/\fin$.

The free sequence number is closely related to other
well known cardinal characteristics of the continuum.
Let us give a brief overview of the relevant definitions.

Let $\mathcal U$ be a non-principal ultrafilter on $\omega$.
The character $\chiup(\mathcal U)$ of $\mathcal U$ is the minimal cardinality of
a base of $\mathcal U$, the $\piup$-character $\piup\chiup(\mathcal U)$
is the minimal cardinality of a $\piup$-base\footnote{
	$\mathcal B \subset {[\omega]}^\omega$ is a $\piup$-base of $\mathcal U$
	if there exists some $B \in \mathcal B$, $B \subset^* U$ for each $U \in \mathcal U$.
	A~$\piup$-base $\mathcal B$ is a base of $\mathcal U$ if moreover $\mathcal B \subset \mathcal U$.
}
of $\mathcal U$. 
The ultrafilter number $\mathfrak u$ is the cardinal
$\min \set{ \chiup(\mathcal U) \mid \mathcal U \text{ is a non-principal ultrafilter on } \omega }$,
the reaping number $\mathfrak r$ is the cardinal
$\min \set{ \piup\chiup(\mathcal U) \mid \mathcal U
		\text{ is a non-principal ultrafilter on } \omega }$.
We opted for a nonstandard definition of the reaping number
as it is more suitable for our purposes.

\begin{theorem}[\cite{balcar-simon}]
	The reaping number $\mathfrak r$ as defined above is equal to
	the minimal cardinality of a family
	$\mathcal R \subset {[\omega]}^\omega$
	such that for each $x \subset \omega$ there is $r \in \mathcal R$
	such that $r \subset^* x$ or $r \cap x =^* \emptyset$.
\end{theorem}

We also need a variant of the ultrafilter number,
let $\mathfrak u^*$ be the cardinal
$\min \set{ \chiup(\mathcal U) \mid \mathcal U
		\text{ is a non-principal ultrafilter such that
		} \chiup(\mathcal U) = \piup\chiup(\mathcal U) }$.\\
The existence of an ultrafilter satisfying
$\chiup(\mathcal U) = \piup\chiup(\mathcal U)$ is unclear in general, 
if no such ultrafilter exists, we declare $\mathfrak u^*$ to be the continuum~$\mathfrak c$.
Bell and Kunen~\cite{bell-kunen} proved that there is always an ultrafilter $\mathcal U$ 
such that $\piup\chiup(\mathcal U) = \cof \mathfrak c$, 
therefore the following question is open only in case the continuum is a singular cardinal.

\begin{question}
	Does \textsf{ZFC} imply the existence of an ultrafilter $\mathcal U$
	such that $\chiup(\mathcal U) = \piup\chiup(\mathcal U)$?
\end{question}

\begin{observation}
	$\mathfrak{r} \le \mathfrak{u} \le \mathfrak{u}^*$.
	If $\mathfrak{r} = \mathfrak{u}$, then $\mathfrak{u}^* = \mathfrak{u}$.
\end{observation}

We say that $\mathcal X \subset {[\omega]}^\omega$ is an independent system
if for every function $f \colon \mathcal X \to 2$ is the family
$\set{ a^{f(a)} \mid a \in \mathcal X }$ centered.
An independent system is maximal if it is maximal with respect to inclusion.
The independence number $\mathfrak i$ is the minimal cardinality
of a maximal independent system.
Although the definitions of a maximal independent system and a maximal free sequence
are somewhat similar, we know very little about the relations between these objects
and the relation between the cardinal characteristics $\mathfrak i$ and $\mathfrak f$.

A strictly $\subset^*$-decreasing sequence in ${[\omega]}^\omega$ is always a free sequence.
Maximal such decreasing sequences (with respect to end-extension) are called towers,
the smallest cardinality of a tower is the tower number~$\mathfrak t$.
A tower does not need to be a maximal free sequence.
On the other hand if a free sequence generates an ultrafilter,
then it is maximal.
This observation allows us to deduce that there are maximal free sequences
of ordinal length $\omega_1$ in the Miller model as it contains such towers
which generate ultrafilters~\cite{Miller-rational}.
In particular, the Miller model demonstrates the consistency of
$\omega_1 = \mathfrak u = \mathfrak f < \mathfrak i = \mathfrak c = \omega_2$.

\begin{question}
	Is $\mathfrak{i} < \mathfrak{f}$ consistent with \textsf{ZFC}?
\end{question}

The first part of the following proposition is already in~\cite{monk}. 

\begin{proposition}
	$\mathfrak{r} \le \mathfrak{f} \le \mathfrak{u}^*$
\end{proposition}
\begin{proof}
	First assume that $A$ is a free sequence of size smaller than $\mathfrak r$.
	Let $\mathcal U$ be a non-principal ultrafilter extending $A$,
	$\comb{A}$ is not a $\piup$-base of $\mathcal U$ as it is of size $< \mathfrak r$.
	Choose $a \in \mathcal U$ such that $a^0 \cap c$ is infinite for each $c \in \comb{A}$.
	Now $A\conc \la a \ra$ is a free sequence and the first inequality is proved.

	Assuming $\mathfrak u^* < \mathfrak c$, let $\set{ u_\alpha \mid \alpha < \mathfrak u^* }$
	be a base of an ultrafilter $\mathcal U$ such that
	$\piup\chiup(\mathcal U) = \chiup(\mathcal U)$.
	Using induction on $\alpha$ we can define a free sequence
	$\la a_\alpha \mid \alpha < \mathfrak u^* \ra$.
	Start with $a_0 = u_0$. If $A_\beta = \la a_\alpha \mid \alpha < \beta \ra$
	is defined, use $\card{\comb{A_\beta}} < \piup\chiup(\mathcal U)$ to find
	$b_\beta \in \mathcal U$ such that $b^0 \cap c$ is infinite for each $c \in \comb {A_\beta}$.
	Let $a_\beta = b_\beta \cap u_\beta$, notice that $A_\beta \conc \la a_\beta \ra$
	is a free sequence.
	Finally, the constructed free sequence is a base of the ultrafilter $\mathcal U$
	and hence it is maximal.
\end{proof}

\begin{corollary}\label{cor:if-r=u}
	If $\mathfrak{r} = \mathfrak{u}$, then $\mathfrak{f} = \mathfrak{u} = \mathfrak{r}$.
\end{corollary}

\begin{question}
	Is $\mathfrak{r} < \mathfrak{f}$ consistent with \textsf{ZFC}?
	What about $\mathfrak{u} < \mathfrak{f}$?
\end{question}

The natural candidate for a model satisfying $\mathfrak{r} < \mathfrak{f}$
is the model constructed in~\cite{goldstern-shelah}.
Corollary~\ref{cor:if-r=u} presents a substantial obstacle
when constructing a model where $\mathfrak u < \mathfrak f$.
In such model necessarily $\mathfrak r < \mathfrak u < \mathfrak f$ holds,
and this cannot be achieved using the usual technique
of countable support forcing iteration.

The next proposition generalizes a property of decreasing sequences to arbitrary free sequences.

\begin{proposition}\label{prop:notUf}
	Let $A = \la a_\alpha \mid \alpha < \gamma \ra$ be a free sequence
	and $\cf \gamma < \mathfrak t$. Then the free sequence $A$ does not generate an ultrafilter.
\end{proposition}
\begin{proof}
	Let $\la \gamma_i \mid i \in \cf \gamma \ra$ be a sequence of ordinals
	cofinal in $\gamma$. For $ i \in \cf \gamma$ choose an ultrafilter  $\mathcal U_i$
	extending the centered family
	$\set{ a_\alpha \mid \alpha < \gamma_i } \cup \set{ a_\alpha^0 \mid \gamma_i \leq \alpha < \gamma }$.
	If $\mathcal U$ is an ultrafilter extending $A$,
	then $\la \mathcal U_i \mid i \in \cf \gamma \ra$
	is a sequence in the Stone space
	converging to $\mathcal U$, a contradiction with $\cf \gamma < \mathfrak t$.
\end{proof}

In fact, the same argument can be used to prove to following,
presumably well known fact.

\begin{observation}
	Let $\mathcal X$ be an independent system and $f \colon \mathcal X \to 2$ any function.
	Then $\mathcal X_f = \set{ a^{f(a)} \mid a \in \mathcal X }$ does not generate an ultrafilter.
\end{observation}
\begin{proof}
	If $\mathcal X$ is finite, the statement is trivially true. 
	If $\mathcal X$ is infinite, then $\mathcal X_f$ can be ordered
	with an order type of cofinality $\omega$, and then use Proposition~\ref{prop:notUf}.
\end{proof}

The maximal free sequences constructed so far generate ultrafilters.
The next proposition shows an elementary example demonstrating
that this does not need to be the case for a general free sequence.

\begin{proposition}\label{prop:nonUF}
	For any given maximal free sequence
	there exists a maximal free sequence of the same cardinality
	which does not generate an ultrafilter.
\end{proposition}
\begin{proof}
	We can assume that $\omega = X \cup Y$ for $X, Y$ infinite disjoint,
	and there are maximal free sequences
	$A = \la a_\alpha \subset X \mid \alpha \in \gamma\ra$
	and $B = \la b_\alpha \subset Y \mid \alpha \in \gamma\ra$
	in $\pw{X}$ and $\pw{Y}$ respectively.
	For $\la \alpha, i \ra \in \gamma \times 2$
	let $c_{\alpha, i} = a_{\alpha} \cup b_{\alpha+i}$.
	Considering the lexicographical order on $\gamma \times 2$
	we get a sequence $C = \la c_{\alpha, i} \mid \la\alpha, i\ra \in  \gamma \times 2\ra $.
	This sequence does not generate an ultrafilter as both $X$ and $Y$
	are positive with respect to the filter the sequence generates.
	We claim that $C$ is a maximal free sequence on $\omega$.
	Checking that $C$ is a free sequence is straightforward.
	To verify the maximality, take any $z \subset \omega$.
	If $z$ is not positive with respect to both the filters
	generated by $A$ and $B$, then $C\conc\la z \ra$ is not centered.
	Assume $z$ is positive with respect to the filter
	generated by $A$.  As $A$ is maximal, there are
	$\Gamma < \Delta \in {[\gamma]}^{<\omega}$ such that
	$\set{a_\alpha \mid \alpha \in \Gamma} \cup \set{a^0_\alpha \mid \alpha \in \Delta}
		\cup \set{ z^0 \cap X}$ has only finite intersection.
	We may moreover suppose that there is $\alpha \in \Gamma$
	such that $\alpha+1 \in \Delta$.
	As the intersection of $\set{c_\alpha \mid \alpha \in \Gamma\times 2}
		\cup \set{c^0_\alpha \mid \alpha \in \Delta \times 2}$ is a
	subset of $X$, it has only finite intersection with $z^0$ and
	$C$ cannot be end-extended by $z$.
	The reasoning when $z$ is positive with respect to $B$ is analogous.
\end{proof}

Regarding the proof Proposition~\ref{prop:nonUF},
if the free sequences $A$ and $B$ generate ultrafilters,
we can use similar construction, defining a free sequence
$C = \la c_\alpha = a_{\alpha} \cup b_{\alpha} \mid \alpha \in \gamma \ra \conc \la X \ra $.
This way we get an example of a maximal free sequence such that the order type of
$C$ is not a limit ordinal.

\section{Towards $\mathfrak{i} = \mathfrak{f} < \mathfrak{u}$}\label{sec:prelim}

The rest of the paper is focused on proving that
$\mathfrak{f} < \mathfrak{u}$ is consistent with \textsf{ZFC}.
The model where this holds is the model for $\mathfrak{i} < \mathfrak{u}$
due to Shelah~\cite{shelah-i<u}.
As the original paper is not easy to digest,
we opted to include
the proof. Our original contribution here is only the proof
that $\mathfrak{i} = \mathfrak{f}$ in this model.

Let us start with reviewing some basic terminology and folklore knowledge.
An ideal on $\omega$ is a set $\mathcal I \subset \pw{\omega}$,
such that if $I, J \in \mathcal I$ and $A \subset I$,
then $A \in \mathcal I$ and $I \cup J \in \mathcal I$.
The ideal $\mathcal I$ is proper if $\omega \notin \mathcal I$.
All ideals considered here will be proper ideals on $\omega$
containing all finite subsets of $\omega$.
A filter will generally be a dual of such ideal.
For an ideal $\mathcal I$ we denote the dual filter as $\mathcal I^*$.
We say that $\mathcal K$ is a \emph{co-filter} if $\pw{\omega} \setminus \mathcal K$
is a filter.

For a filter base $\mathcal H \subset \pw{\omega}$ we denote $\sq{\mathcal H}$ the filter
generated by $\mathcal H$,
i.e.\ $F \in \sq{\mathcal H}$ iff $H \subset^* F$ for some $H \in \mathcal H$.
We use the same notation for co-filters generated by a co-filter base,
the intended meaning of the notation should be always apparent from the context.
We will need a folklore classification of filters.
For $A \subseteq \omega$ we denote
$\epsilon_A \colon \omega \to A$ the unique increasing surjection, and
$\overline\epsilon_F \in \omega^\omega$ the function
$\overline\epsilon_F \colon n \mapsto \epsilon_F(n+1) - \epsilon_F(n)$.
Filter $\mathcal F$ is non-meager if the family
$\set{\epsilon_F \mid F \in \mathcal F}$ is unbounded in $(\omega^\omega, <^*)$.
Filter $\mathcal F$ is rare\footnote{
	Rare filters are also called Q-filters.
	We opted for the original terminology of Choquet.
}
if the family
$\set{\overline \epsilon_F \mid F \in \mathcal F}$ is dominating.
Filter $\mathcal F$ is a P-filter if for each $C \in {[\mathcal F]}^\omega$ there
exists $F \in \mathcal F$ such that $F \subset^* X$ for each $X \in C$.

We will use the following standard \emph{diagonal} properties of these filters.

\begin{fact}\label{fact:non-meager}
	Filter $\mathcal F$ is a non-meager P-filter if and only if for each sequence
	$\set{F_n \in \mathcal F \mid n \in \omega}$ there exist $F \in \mathcal F$
	such that $F \setminus n \subset F_n$ for infinitely many $n \in \omega$.
\end{fact}

Notice that the condition in the preceding fact can be equivalently formulated as
``$F \setminus (n+1) \subset F_n$ for infinitely many $n \in F$.''

\begin{fact}\label{lem:diagonal}
	Filter $\mathcal F$ is a rare P-filter iff $\mathcal F$ has the \emph{diagonal property},
	i.e.\ for each $\set{ F_n \in \mathcal F \mid n \in \omega }$ there exists $F \in \mathcal F$
	such that $F \setminus (n+1) \subseteq F_n$ for each $n \in F$.
\end{fact}

We say that a forcing notion $P$ is bounding if for every $P$-generic extension $V[G]$
and each $f \in \omega^\omega \cap V[G]$ there is $g \in \omega^\omega \cap V$ such that
$f \leq g$.
Forcing $P$ has the Sacks property if for each $f \in \omega^\omega \cap V[G]$
there exists a sequence $\set{G_n \mid \card{G_n} \leq 2^n, n \in \omega} \in V$
such that $f(n) \in G_n$ for each $n \in \omega$.
We can equivalently require $\card{G_n} \leq n+1$, see e.g.~\cite{SacksForcingandtheSacksProperty}. 
Every forcing with the Sacks property is bounding.
Every rare or non-meager filter generates a filter with the same property in every
generic extension via a bounding forcing.
Every P-filter generates a P-filter in a generic extension via a proper forcing.

We will use the standard notation for the Cohen poset, the set
$\mathbf C_\kappa = \set{ h \colon \kappa \to 2 \mid \card{h} < \omega }$ ordered by
reverse inclusion. If $\kappa = \omega$, we write just $\mathbf C$.
A set $D \subset \mathbf C_\kappa$ is dense if for each
$h \in \mathbf C_\kappa$ there exists $g \in D$, $g \supset h$.
For dense sets $C, D \subset \mathbf C_\kappa$ we say that
$C$ \emph{refines} $D$ if for each $h \in C$ there exists $g \in D$ such that $g \subseteq h$.
If $W$ is an extension of a model of set theory $V$, we say that $W$ is \emph{Cohen-preserving}
if for each dense $D \subset \mathbf C$, $D \in W$
exists $C \in V$ which refines $D$.
We say that a forcing is Cohen-preserving if every generic extension via this forcing
is Cohen-preserving.
Although this property of forcing notions is considered in the literature,
e.g.~\cite[6.3.C]{bartoszynski-judah}, there does not seem to be a unified terminology.

The following proposition is implicitly proved in~\cite{miller-properties}. 
We learned both the proposition and the proof from O.\ Guzm\'{a}n.
We reproduce the proof for the sake of completeness.


\begin{proposition}
	If a forcing notion has the Sacks property,
	then it is Cohen-preserving.
\end{proposition}
\begin{proof}
	Suppose that $V[G]$ is a generic extension via a forcing which has
	the Sacks property, let $D \in V[G]$ be an open dense subset of $\mathbf C$.
	We will, without loss of generality, work with $2^{<\omega}$ instead of $\mathbf C$.
	As the extension is bounding,
	there is $f \colon \omega \to \omega$ in $V$
	such that for each $n \in \omega$ there is $s \in 2^{f(n)}$ 
	such that $t\conc s \in D$ for each $t \in 2^n$. 
	Fix a dense subset $\set{t_n \mid n \in \omega}$ of $2^{<\omega}$ in $V$
	such that $\card{t_n} = n$.

	In $V[G]$ define a function $h\colon \omega \to {[2^{<\omega}]}^{<\omega}$
	such that $\card{h(n)} = n+1$ for each $n \in \omega$.
	The function is defined in the following way.
	Given $n \in \omega$ let $r_n(0) = n$.
	When $r_n(i)$ for $i \leq n+1$ is defined, choose $s^i \in 2^{f(r_n(i))}$
	such that $x \conc s^i \in D$ for each $x \in 2^{r_n(i)}$
	and let $r_n(i+1) = r_n(i) + f(r_n(i))$.
	Finally let $h(n) = \la s^i \mid i \leq n \ra$.
	As the extension has the Sacks property,
	there is a sequence $\la H(n) \subset {[2^{<\omega}]}^{n+1} \mid n \in \omega \ra$ in $V$
	such that $\card{H(n)} = n+1$ and $h(n) \in H(n)$ for each $n \in \omega$.
	We denote $H(n) = \la S_k(n) \mid k \leq n \ra$ and
	$S_k(n) = \la s^i_k(n) \mid i \leq n \ra$.
	We may assume that $\card{s^i_k(n)} = f(r_n(i))$ for
	each $k,i \leq n$, $n \in \omega$.
	
	Finally let $z_n = t_n \conc s^0_0(n) \conc s^1_1(n) \conc \ldots \conc s^{n}_{n}(n)$.
	The set $C = \set{ z_n \mid n \in \omega} \in V$ is obviously dense,
	and $C \subset D$ because for each $n\in \omega$ there is
	$k$ such that $s^k_k(n) \in h(n)$.
\end{proof}

Since the posets $\mathbf C_\kappa$ are c.c.c.,
being Cohen-preserving already guarantees
an analogous property for these posets as well.

\begin{lemma}\label{lem:om1_dense}
	Let $P$ be a proper Cohen-preserving forcing, $G$ a generic filter on~$P$.
	For each $\kappa$ and each dense $D \subset \mathbf C_\kappa$ in $V[G]$, there exists
	$C \in V$ refining~$D$.
\end{lemma}
\begin{proof}
	Since $\mathbf C_\kappa$ is c.c.c.\ there is a countable set $a \in V[G]$, $a \subset \kappa$,
	and a countable dense $D' \subseteq D$, $D'\subset \mathbf C_a$.
	Since $P$ is proper, there exist a countable $b \in V$ such that $a \subset b$, i.e.\
	$D' \subset \mathbf C_b$.
	As $P$ is Cohen-preserving, there exists $C \in \mathbf C_b \cap V$ refining $D'$,
	hence also refining~$D$.
\end{proof}

\section{The forcing notion}\label{sec:forcing}

We will say that $E = \set{e_k \subset \omega \mid k \in \omega}$ is a
partition if $e_k \cap e_j = \emptyset$ for $k \neq j$.
We will usually deal with infinite partitions
and we always assume $\min e_k < \min e_j$ for $k < j$.
We denote $\dom E = \bigcup E$.
Partition $D = \set{ d_k \mid k \in\omega }$
is coarser than $E$
if each element of $D$ is a union of elements of $E$.
We say that $D$ is cruder than $E$
if $D \restriction \dom E = \set{ d_k \cap \dom E \mid k \in \omega }$ 
is coarser than $E$ and for each $d \in D$ there exists $e \in E$ 
such that $\min d = \min e$.
If $\mathcal I$ is an ideal on $\omega$, we say that $E$ is an
\emph{$\mathcal I$-partition} if $e_k \in \mathcal I$ for each $k \in \omega$
and $\dom E \in \mathcal{I}^*$.

For the purpose of this paper a tree $T$ is an initial subtree of the tree of finite
$0$\fdash$1$ sequences $\left(2^{<\omega} , \subseteq \right)$ with no maximal elements (leaves).
For $t \in T$ we denote $T[t]$ the subtree
consisting of all nodes of $T$ compatible with $t$.
For $n \in \omega$ we denote by $T^{(n)}$ the set of
all nodes $t \in T$ such that $\card{t} = n$ (i.e.\ the nodes from the $n$-th level).
A node $t \in T$ is a \emph{branching node} of $T$ if
both $t \conc 0 \in T$ and $t \conc 1 \in T$.
We say that the $n$-th level is a \emph{branching level}
if each element of $T^{(n)}$ is a branching node.
We say that a tree is \emph{uniformly branching}
if each branching node is an element of a branching level.

Given a tree $T$ we say that the level $m$ \emph{depends} on a level $n$ if
$n \leq m$, $n$ is a branching level, and for each
$s, t \in T^{(m+1)}$ is $s(m) + s(n) = t(m) + t(n) \bmod 2$.
We call such levels $m$ \emph{dependent} levels,
levels which are not dependent are \emph{independent}.
Note that for a given dependent level $m$ there is a unique $n$ such that $m$ depends on $n$,
and each branching level depends on itself.
We say that a level is independent if it does not depend on any level.
To each uniformly branching tree $T$ we assign a partition 
denoted $E^T = \set{ e^T_k \mid k \in \omega }$ such that
if $m$ and $n$ are dependent levels, then $m$ and $n$ are in the same element of $E^T$ iff
$m$ and $n$ depend on the same level, and $\dom E^T$ is exactly the set of all dependent levels.
The superscripts will occasionally be omitted if clear from the context.
Let $\mathcal I$ be an ideal on $\omega$,
we say that a tree $T$ is
\emph{$\mathcal I$-suitable} if 
$T$ is uniformly branching and
$E^T$ is an $\mathcal I$-partition.
The poset of $\mathcal I$-suitable trees ordered by inclusion
will be denoted $\mathbf Q_{\mathcal I}$.
Note that for $S < T \in \mathbf Q_{\mathcal I}$
dependent levels of $T$ can in general be independent levels of $S$,
and independent levels of $T$ can become dependent levels in $S$.
Thus $S < T$ does not necessarily imply that $E^S$ is coarser than $E^T$,
on the other hand $E^S$ is cruder than $E^T$.

This poset is sometimes called the \emph{party forcing}.\footnote{
	Organizing a party in the Hilbert hotel is a difficult task,
	guests may or may not like their lesser colleagues.
}
This version of the forcing is slightly different
than the one used in~\cite{shelah-i<u}, the conditions of the poset
used by Shelah did explicitly remember the partitions $E^T$.
Nevertheless, our version of the poset works in the same way.
This type of forcing was also recently used by Guzm\'{a}n~\cite{osvaldo} 
to prove that the homogeneity number $\mathfrak{hm}$ 
can be consistently smaller than $\mathfrak{u}$.

For $T \in \mathbf Q_{\mathcal I}$ and a partial function $f \colon \omega \to 2$
we denote by $T_f$ the largest subtree of $T$ with the property that
if $k \in \dom f$, $n \in e^T_k$, $n$ is a branching level of $T$
(i.e.\ $n = \min e^T_k$), and $t \in T^{(n)}$,
then $t \conc i \in T_f$ only if $f(k) = i$.
Note that $f$ being finite is a sufficient condition guaranteeing
$T_f \in \mathbf Q_{\mathcal I}$.

The forcing will be used to destroy a given ultrafilter,
when we use the dual ideal as a parameter,
the generic real will witness that the ultrafilter does not generate
an ultrafilter in the generic extension.

\begin{lemma}\label{lem:kill}
	Let $\mathcal I$ be a proper ideal on $\omega$ and
	let $G$ be a $\mathbf Q_{\mathcal I}$-generic filter.
	Then $r = \bigcup \bigcap  G  \in 2^\omega$ and
	$r \notin \la \mathcal I \ra \cup \la \mathcal I^* \ra$.
\end{lemma}
\begin{proof}
	The first part of the lemma is immediate since $\mathcal I$
	extends the Fr\'echet ideal.
	Let $T \in \mathbf Q_{\mathcal I}$ be a condition and $I \in \mathcal I$.
	Pick any integer $n \in \dom E^T \setminus I$,
	hence $n \in e_k^T$ for some $k\in \omega$.
	Put $f_i \colon \set{ k } \to 2$, $f_i \colon k \mapsto i$ for $i \in 2$.
	For both $i \in 2$ the conditions $T_{f_i} \in \mathbf Q_{\mathcal I}$ decide whether $n \in r$,
	and they do so in opposite ways.
	That is at least one of them forces that $r \not\subset I$.
	The argument for $r \notin \la \mathcal I^* \ra$ is analogous.
\end{proof}

Let $a \subset \omega$. Suppose that $S < T$ are conditions in $\mathbf Q_{\mathcal I}$
such that for each $k \in a$, if $n$ is the splitting level of $T$ in $e^T_k$,
then $n$ is also a splitting level of $S$
(i.e.\ $a$-th splitting levels are preserved).
We will denote this relation by $S <_a T$.

\begin{lemma}\label{lem:fix_bran-decide}
	Let $T \in \mathbf Q_{\mathcal I}$ be a condition,
	$x$ a name for an element of $V$, and $n \in \omega$.
	There exists a condition $S <_n T$ 
	such that for each $f \in \prescript{n}{}{2}$
	the condition $S_f$ decides the value of $x$.
\end{lemma}
\begin{proof}
	Fix an enumeration	$\prescript{n}{}{2} = \set{ f_i \mid i \in 2^n }$,
	denote $T^0 = T$, and for $i \in 2^n$ repeat the following procedure.

	Suppose that $T^i <_n T$ is defined.
	Find a condition $S^i < T^i_{f_i}$ and $y_i \in V$ such that
	$S^i \Vdash x = y_i$.
	Then let $T^{i+1}$ be the largest subtree of $T^i$ in $\mathbf Q_{\mathcal I}$ such that
	$ T^{i+1}_{f_i} = S^i$.
	Note that the first $n$ many splitting levels of 
	$T^{i+1}$ are the same as the first splitting levels of $T^i$, 
	$T^{i+1} <_n T^i <_n T$.

	Finally let $S = T^{2^n}$.
	Then $S <_n T$ and $S_{f_i} \Vdash x = y_i$ for every $i \in 2^n$.
\end{proof}



Before proving the properness of the forcing $\mathbf Q_{\mathcal I}$
we introduce a game with $\mathcal I$-partitions $\mathrm{PG}(\mathcal I)$.
Player~I starts the game by choosing an $\mathcal I$-partition
$E^0$ 
and then players~I and~II alternate
in building a sequence of $\mathcal I$-partitions.
In round~$n$ player~II plays an $\mathcal I$-partition
$D^n$ 
coarser than $E^n$,
puts $\Delta_n = \dom E^n \setminus \dom D^n$,
and in the next round
player~I replies with an $\mathcal I$-partition
$E^{n+1}$ 
coarser than $D^n$.
After $\omega$ many rounds player~I wins iff
$r = \bigcup \set{ \Delta_n \mid n \in \omega} \in \mathcal I$.

\begin{lemma}
	Player~I has no winning strategy in the game $\mathrm{PG}(\mathcal I)$.
\end{lemma}
\begin{proof}
	If player~I has a winning strategy, then he also
	has a winning strategy such that moreover
	$\bigcap \set{ \dom E^n \mid n \in \omega} = \emptyset$
	(where $\sq{E^n}$ is the sequence of moves of player~I).
	Assuming player~I uses this strategy,
	player~II will play simultaneously two matches of the game
	$\mathrm{PG}(\mathcal I)$.
	He passes his first move in the first match and then he
	always imitates the moves of player~I
	in the other game. This produces results $r$, $r'$ of the two matches
	such that $r \cup r' = \dom E^0 \in \mathcal I^*$.
	Thus in at least one of the two matches player~II won.
\end{proof}

\begin{proposition}\label{prop:proper}
	Let $\mathcal I$ be a maximal ideal.
	The forcing $\mathbf Q_{\mathcal I}$ is proper and has the Sacks property.
\end{proposition}
\begin{proof}
	We will prove both statements simultaneously.
	Let $T \in \mathbf Q_{\mathcal I}$ be a condition and
	$g$ a name for a function in $\omega^\omega$.
	Let $\theta$ be large enough and fix a countable elementary submodel $M \prec H(\theta)$
	such that $\mathbf Q_{\mathcal I}, T, g \in M$.
	Enumerate all $\mathbf Q_{\mathcal I}$-names for ordinals in $M$
	as $\set{\sigma_n \mid n \in \omega}$.
	We will construct a condition $Q < T$
	such that for each $f \in \prescript{n}{}{2}$, $n \in \omega$ the condition
	$Q_f$ decides the value of $g(n)$, and forces $\sigma_n$ to be some element of $M$.
	This will prove the proposition.

	Two players will play the game $\mathrm{PG}(\mathcal I)$ in the model $M$,
	player~I will attempt
	to construct the desired condition during the course of the game.
	Player~I starts by finding a condition $T_0 < T$, $T_0 \in M$ which
	decides $g(0)$ and $\sigma_0$.
	His first move in the game is $E^0 = E^{T_0}$, and the reply of player~II is
	an $\mathcal I$-partition $D^0$.

	Suppose that in the $n$-th round of the game, condition $T_n$ was defined 
	such that $\set{e^{T_n}_k \mid n \leq k \in \omega}$ was cruder than $D^{n-1}$,
	and player~II
	played an $\mathcal I$-partition $D^n$ coarser than $E^n =
		\set{e^{T_n}_k \mid n \leq k \in \omega} \restriction \dom D^{n-1}$ 
	(put $D^{-1} = \omega$).
	Note that these assumptions imply that	
	$\min e$ is a branching level of $T_n$ for each $e \in E^n$.
	Moreover assume that $e^{T_n}_k \cap \dom D^{n-1} = \emptyset$ for $k < n$.
	\begin{claim*}
		There is a condition $T'_n <_{n} T_n$ in $M \cap \mathbf Q_{\mathcal I}$ 
		such that
		\begin{itemize}
			\item $e^{T'_n}_k \cap \dom E^{T_n} = e^{T_n}_k$ for $k < n-1$,
			\item $e^{T'_n}_{n-1} \supseteq e^{T_n}_{n-1} \cup \Delta_n$, and
			\item $\set{e^{T'_n}_k \mid k \geq n}$ is cruder than $D^n$.
		\end{itemize}
	\end{claim*}
	\begin{claimproof}
		To get $T'_n$ work in $M$ and prune the tree $T_n$ in the following way. 
		Preserve the first $n$ many branching levels of $T_n$, 
		note that these levels are not elements of $\dom D^n$.

		If $\ell \in \Delta_n$ is a branching level of $T_n$, 
		make level $\ell$ in the pruned tree depend 
		on the branching level $\min e^{T_n}_{n-1}$.

		If $\ell \in d \in \dom D^n$ is a branching level of $T_n$, 
		make level $\ell$ depend on the level $\min d$. 
		Note that $\min d$ is a branching level of $T_n$ since 
		$D^n$ is coarser than $E^n$.
		It is straightforward to check that the pruned tree 
		fulfills the conditions required of $T'_n$.
	\end{claimproof}\claimdone\medskip
	In the $(n+1)$-th round
	player~I uses Lemma~\ref{lem:fix_bran-decide}
	to find a condition $T_{n+1} <_n T'_n$ in model $M$ such that;
	\begin{itemize}
		\item ${\mleft(T_{n+1}\mright)}_f$ decides $g(n+1)$
		      for each $f \in \prescript{n+1}{}{2}$, and
		\item ${\mleft(T_{n+1}\mright)}_f$ decides $\sigma_{n+1}$
		      for each $f \in \prescript{n+1}{}{2}$.
	\end{itemize}
	Then he passes the $\mathcal I$-partition
	$E^{n+1} = \set{e^{T_{n+1}}_k \mid n+1 \leq k \in \omega} \restriction \dom D^{n}$
	to player~II and awaits his response $D^{n+1}$.
	Note that $\set{e^{T_{n+1}}_k \mid n+1 \leq k \in \omega}$ is cruder than $D^{n}$ 
	and $e^{T_{n+1}}_k \cap \dom D^{n} = \emptyset$ for $k < n+1$, 
	so player~I can continue using the strategy described above to choose his next move. 

	This strategy is not winning for player~I, so we can assume that the
	game is played so that player~II wins, i.e.\
	$r = \bigcup \set{ \Delta_n \mid n \in \omega} \in \mathcal I^*$
	(the ideal $\mathcal I$ is maximal).

	Once the game is over, define $Q = \bigcap \set{T_n \mid n \in \omega}$.
	Notice that for $e^Q_n \in E^Q$, $e^Q_n \cap r = \Delta_n$, and
	$r \subset \dom E^Q$. 
	Thus $\dom E^Q \in \mathcal{I}^*$,
	$E^Q$ is an $\mathcal I$-partition and $Q \in \mathbf Q_\mathcal{I}$.
	Since $Q <_{n} T_n$ for each $n \in \omega$, $Q$ is the desired condition.
\end{proof}

\begin{corollary}
	The poset $\mathbf{Q}_\mathcal{I}$ is a Cohen-preserving forcing notion.
\end{corollary}

The proof of Proposition~\ref{prop:proper} gives us in fact the following.
\begin{corollary}\label{cor:nice_name}
	Let $T \in \mathbf Q_{\mathcal I}$ be a condition and let $X$ be a name for a subset of $\omega$.
	There is a condition $S < T$ such that for each $n \in \omega$ and $f \in 2^{n+1}$,
	$S_f$ forces either $n \in X$ or $n \notin X$.
\end{corollary}

\section{Dense independent systems}\label{sec:mis}

Let $\mathcal A \subset \pw{\omega}$ be an independent system.
Remember that the set of finite partial functions $\set{ h \colon \mathcal A \to 2 }$
is denoted $\mathbf C_{\mathcal A}$,
and it carries the usual inclusion order.
For each $h \in \mathbf C_{\mathcal A}$ we put
$\mathcal A^h = \bigcap \set{ A^{h(A)} \mid A \in \dom h} \in {[\omega]}^\omega$.
For $X \subseteq \omega$ we will say that $h \in \mathbf C_{\mathcal A}$
\emph{reaps} $X$ if either $\mathcal A^h \subset^* X$ or $\mathcal A^h \cap X =^* \emptyset$.
If the first option $\mathcal A^h \subset^* X$ occurs,
we say that $h$ \emph{hits} $X$.
The independent system $\mathcal A$ is maximal
iff the set $\set{ h \mid h \text{ reaps } X }$
is nonempty for  each $X \subseteq \omega$.

We say that the independent system $\mathcal A$ is \emph{dense}
if the set $\set{ h \mid h \text{ reaps } X }$ is dense in $\mathbf{C}_{\mathcal A}$
for each $X \subseteq \omega$.
It is easy to see that every dense independent system is maximal.
Dense independent systems were originally introduced in~\cite{goldstern-shelah}
and recently studied in~\cite{vera-diana}.
For each maximal independent system $\mathcal A$
there exists $h \in \mathbf C_{\mathcal A}$ such that
$\mathcal A \restriction \mathcal A^{h} = 
\set{A \cap \mathcal A^{h} \mid A \in \mathcal A \setminus \dom{h}}$ is a dense
independent system, see~\cite[Lemma 6.6, 6.7]{goldstern-shelah}.

Denote by $\mathcal D$ the collection of dense subsets of $\mathbf{C}_{\mathcal A}$.
The filter on $\omega$ generated by sets of form
$F(D) = \bigcup \set{ \mathcal A^{h} \mid h \in D }$ for some $D \in \mathcal D$
will be denoted $\mathscr F_{\mathcal A}$.

\begin{lemma}\label{lem:Fa_filter-equivalence}
	Let $\mathcal A$ be an independent system and let $X$ be a subset of $\omega$.
	The set $\set{ h \mid h \text{ hits } X }$
	is dense in $\mathbf{C}_{\mathcal A}$ 
	if and only if $X \in \mathscr F_{\mathcal A}$.
\end{lemma}
\begin{proof}
	If $\mathcal A$ is finite, the proof is straightforward, 
	therefore assume $\mathcal A$ is infinite. 
	The `if' implication follows directly from the definition of $\mathscr F_{\mathcal A}$.
	\begin{claim*}
		For each $h \in \mathbf C_{\mathcal A}$ and $n \in \omega$ 
		there exists $h' \in \mathbf C_{\mathcal A}$, $h \subseteq h'$ 
		such that $\mathcal A^{h'} \subset \mathcal A^{h} \setminus n$.
	\end{claim*}
	\begin{claimproof}
		Let $\set{ A_k \mid k \in n }$ be a subset of $\mathcal A$ disjoint with $\dom h$. 
		Extend $h$ 
		by defining $h' \colon A_k \mapsto i \in 2$ iff $k \notin {A_k}^i$ for $k \in n$.
		Then $h'$ is as required.
	\end{claimproof}\claimdone\medskip
	Suppose that the set $\set{ h \mid h \text{ hits } X }$
	is dense in $\mathbf{C}_{\mathcal A}$. 
	Then the claim implies that $\set{ h \mid \mathcal{A}^h \subset X }$ 
	is also dense and the `only if' implication follows.
\end{proof}

We will denote $\mathscr C_{\mathcal A} =
	\set{ \omega \setminus \mathcal A^h \mid h \in \mathbf{C}_{\mathcal A}}$.
The following observation will be crucial for the preservation of maximality
of a given independent system.

\begin{lemma}\label{lem:dense-equivalence}
	An independent system $\mathcal A$ is dense if and only if the co-filter
	$\pw{\omega} \setminus \mathscr F_{\mathcal A}$
	is generated by the set $\mathscr C_{\mathcal A}$.
\end{lemma}
\begin{proof}
	Suppose that $\mathcal A$ is dense and $X \subseteq \omega$.
	If $\set{ h \mid h \text{ hits } X }$ is dense in $\mathbf{C}_{\mathcal A}$,
	then $X \in \mathscr F_{\mathcal A}$.
	Otherwise there is $h \in \mathbf{C}_{\mathcal A}$
	such that $\mathcal A^h \cap X =^* \emptyset$ 
	and $X \in \mathscr C_{\mathcal A}$.

	To verify the other implication let $X \subseteq \omega$
	and $h \in \mathbf{C}_{\mathcal A}$ be given,
	let $X' = \left( X \cap \mathcal A^h \right) \cup \left( \omega \setminus \mathcal A^h \right)$.
	If $X' \in \mathscr F_{\mathcal A}$,
	then there is $h' \supset h$ such that
	$\mathcal A^{h'} \subset^* X'$, and hence $\mathcal A^{h'} \subset^* X$.
	Otherwise $X' \in \mathscr C_{\mathcal A}$,
	there is $h'$ such that $\mathcal A^{h'} \cap X' =^* \emptyset$.
	Thus $h \subseteq h'$ and $\mathcal A^{h'} \cap X =^* \emptyset$.
\end{proof}

The definition of ${\mathscr C_{\mathcal A}}$ is absolute for all models of set theory. 
The definition of $\mathscr {F_{\mathcal A}}$ behaves well when considering Cohen-preserving extension. 

\begin{lemma}\label{lem:F_pres}
	Let $\mathcal A \in V$ be an independent system and let $W$
	be a Cohen-preserving extension of $V$.
	The filter  $\mathscr {F_{\mathcal A}}^{W}$ is generated by ${\mathscr F_{\mathcal A}}^V$.
\end{lemma}
\begin{proof}
	Follows immediately from Lemma~\ref{lem:om1_dense}.
\end{proof}

\begin{remark}\label{rem:preserve-dense}
	Lemmas~\ref{lem:dense-equivalence} and~\ref{lem:F_pres}
	imply that to prove that a dense independent system $\mathcal A \in V$
	remains dense in a Cohen-preserving extension $W$,
	it is sufficient to demonstrate that in $W$ is 
	$\pw{\omega} = \la{\mathscr F_{\mathcal A}}^V\ra \cup \la {\mathscr C_{\mathcal A}} \ra$.
\end{remark}

\begin{proposition}\label{prop:sel_sys}
	Assume \textsf{CH}. There exists an independent system $\mathcal A$
	with the following properties:
	\begin{enumerate}
		\item\label{item:dense} $\mathcal A$ is dense,
		\item\label{item:F_prop} $\mathscr F_{\mathcal A}$ is a rare P-filter.
	\end{enumerate}
\end{proposition}

We call an independent system satisfying properties
(\ref{item:dense}) and (\ref{item:F_prop}) \emph{selective}.

\begin{proof}
	Enumerate the functions in $\omega^\omega$ as
	$\set{f_\alpha \mid \alpha \in \omega_1, \alpha \text{ limit} }$,
	enumerate maximal antichains in $\mathbf C_{\omega_1}$ as
	$\set{ H_\alpha \mid \alpha \in \omega_1, 0 < \alpha \text{ limit} }$ so that
	$H_\alpha \subset \mathbf C_\alpha$,
	and enumerate all elements of $\pw{\omega} \times \mathbf C_{\omega_1}$
	as $\set{ \la X_\alpha, g_\alpha \ra \mid \alpha \in \omega_1}$
	so that $g_\alpha \in \mathbf C_\alpha$. 

	We proceed by induction, for $\alpha < \omega_1$ we will define
	$\la A_\alpha, B_\alpha \mid \alpha < \omega_1 \ra$
	such that $A_\alpha \subset B_\alpha \subset^* B_\beta \subset \omega$  for $\beta < \alpha$, and
	$\bar{\mathcal A}_{\alpha} =  \la A_\beta \cap B_\alpha \mid \beta  < \alpha\ra$
	is an independent system.
	We write ${\mathcal A}_{\alpha} =  \la A_\beta\mid \beta  < \alpha\ra$.

	Start with $B_0$ such that $f_0 < \overline\epsilon_{B_0}$.
	If $\la A_\alpha, B_\alpha \mid \alpha < \beta \ra$  and $B_\beta$ are defined,
	let $B_{\beta + 1} = B_\beta$ and
	choose any  $A_{\beta} \subset B_\beta$ such that $\bar{\mathcal A}_{\beta+1}$
	is an independent system,
	this is possible since $\bar{\mathcal A}_{\alpha}$ is countable and hence not maximal.
	Moreover, letting
	$Z_\beta = {\bar{\mathcal A}_{\alpha}}^{g_\alpha}$,
	if it is possible to choose $A_{\beta}$
	such that $A_\beta \cap Z_\beta  = X_\beta  \cap Z_\beta$, do so.\footnote{
		We will be slightly abusing the notation, identifying $\mathbf C_{\bar{\mathcal A}_\alpha}$
		with $\mathbf C_{\alpha}$ etc.
	}

	Suppose $\la A_\alpha, B_\alpha \mid \alpha < \beta \ra$
	is defined and 
	$\beta$ is a limit ordinal.
	\begin{claim*}
		There is $B_\beta \subset \omega$ such that $B_\beta \subset^* B_\alpha$
		for $\alpha < \beta$, $f_\beta < \overline\epsilon_{B_\beta}$,
		$B_\beta \subset \bigcup
			\set{{\mathcal A}^h \mid h \in H_\alpha }$,
		and $\bar{\mathcal A}_\beta$ is an independent system.
	\end{claim*}
	\begin{claimproof}
		Fix a sequence $\alpha(n)$ converging to $\beta$ and an enumeration
		$\set{g_i \mid i \in \omega}$ of all
		functions in $\mathbf C_\beta$ which extend some element of
		$H_\alpha$, with infinite repetitions, and
		so that $\dom h_i \subset {\alpha(i)}$ for each $i \in \omega$.

		Since the sets
		$C_i =
			{\mathcal A_\beta}^{g_i}
			\cap \bigcap\set{ B_{\alpha(j)} \mid j \leq i}$
		are infinite for all $i \in \omega$,
		it is possible to choose infinite $B_\beta$ such that
		$\epsilon_{B_\beta}(i) \in C_i$ and
		$f_\beta < \overline\epsilon_{B_\beta}$.
		This is as required since $\bar{\mathcal A}_\beta$ is an independent system.
	\end{claimproof}\claimdone\medskip

	This completes the inductive construction.
	We constructed an independent system
	$\mathcal A = \set{A_\alpha \mid \alpha \in \omega_1}$.
	To check that it is dense take any
	$\la X_\beta, g_\beta \ra \in \pw{\omega} \times \mathbf C_{\omega_1}$.
	If $A_\beta$ was chosen so that
	$A_\beta \cap Z_\beta  = X_\beta  \cap Z_\beta$, we are done.
	If $A_\beta$ was not chosen with this property,
	there is some $g \in \mathbf C_{\beta}$, $g_\beta \subset g$
	such that
	${\mathcal A}^g$
	reaps $X_\beta \cap B_\beta$ and we are also done,
	as we can extend $g$ by declaring $g \colon \beta \mapsto 1$
	to achieve ${\mathcal A}^g \subset B_\beta$.

	The inductive construction ensures that the filter generated by the decreasing tower
	$\mathcal T = \set{ B_\alpha \mid \alpha \in \omega_1 }$
	is a rare P-filter.

	\begin{claim*}
		The filter $\mathscr F_{\mathcal A}$ is the filter
		generated by $\mathcal T$.
	\end{claim*}
	\begin{claimproof}
		For $\alpha \in \omega_1$ let
		$D = \set{h \in \mathbf C_{\mathcal A} \mid \text{ there is } k \in \omega \text{ such that }
				A_{\alpha + k} \in \dom h}$.
		The set $D$ is dense in $\mathbf C_{\mathcal A}$ and $F(D) \subset B_\alpha$
		is witnessing $B_\alpha \in \mathscr F_{\mathcal A}$.

		On the other hand take any dense $D \subset \mathbf C_{\mathcal A}$.
		There is some $\beta \in  \omega_1$ such that $H_\beta \subset D$.
		Since each element of $D$ is compatible with some element of $H_\beta$,
		we have that $F(H_\beta) \subset F(D)$.
		The set $B_\beta$ was chosen so that $B_\beta \subset F(H_\beta)$.
	\end{claimproof}
\end{proof}

\begin{theorem}\label{thm:onestep_preserv}
	Let $\mathcal A$ be a selective independent system and
	let $\mathcal I$ be a maximal ideal.
	If $G$ is a $\mathbf Q_{\mathcal I}$-generic filter,
	then $\mathcal A$ is a selective independent system in $V[G]$.
\end{theorem}
\begin{proof}
	The system $\mathcal A$ remains independent in any extension.
	Since $\mathbf Q_{\mathcal I}$ is Cohen-preserving,
	Lemma~\ref{lem:F_pres} states that, in $V[G]$, the filter
	${\mathscr F_{\mathcal A}}^{V[G]}$ is generated by
	${\mathscr F_{\mathcal A}}^V$.
	Thus the filter ${\mathscr F_{\mathcal A}}^{V[G]}$ is a P-filter
	since $\mathbf Q_{\mathcal I}$ is proper, and
	it is rare since $\mathbf Q_{\mathcal I}$ has the Sacks property.
	To show that $\mathcal A$ remains dense in the extension we will use
	Remark~\ref{rem:preserve-dense}.

	Let $T \in \mathbf Q_{\mathcal I}$ be a condition
	and $X$ a name for a subset of $\omega$.
	Suppose that no stronger condition forces that
	$X \in \la \mathscr C_{\mathcal A} \ra$,
	i.e.\ for each $S < T$, 
	$X_S = \set{ n \in \omega \mid 
			S \not\Vdash n \notin X }
		\in \mathscr F_{\mathcal A}$.
	We will show that such $T$ forces that $X \in \la \mathscr F_{\mathcal A}\ra$.
	In particular, for given $h \in \mathbf{C}_{\mathcal A}$ we find
	$g \supset h$ and $Q < T$ such that $Q \Vdash \mathcal A^{g} \subset^* X$.

	We may assume that for each $n \in \omega$ and $f \in 2^n$
	the condition $T_f$ decides $X \cap n$ (use Corollary~\ref{cor:nice_name}).
	For $n \in \omega$ put
	$X_n = \bigcap \set{ X_{T_f} \mid f \in 2^n } \in \mathscr F_{\mathcal A}$.
	Note that for each $n \in \omega$ and each $k \in X_n \setminus (n+1)$ there is a condition
	$T_n(k) <_{\omega \setminus [n, k)} T$
	such that $T_n(k) \Vdash k \in X$.
	The filter $\mathscr F_{\mathcal A}$ has the diagonal property, i.e.\ there is
	$F \in \mathscr F_{\mathcal A}$ such that $F \setminus (n+1) \subseteq X_n$
	for each $n \in F$.
	Let $\set{ k_n \mid n \in \omega }$ be the increasing enumeration of such an $F$.
	The choice of $F$ ensures that for each $n \in \omega$
	the condition $T_{k_n}(k_{n+1})$ is defined.

	Since $\mathcal A$ is dense, there are $g_0, g_1 \supset h$ such that
	$\mathcal A^{g_0} \cup \mathcal A^{g_1} \subset^* F$, and
	$\mathcal A^{g_0} \cap \mathcal A^{g_1} = \emptyset$.
	For $i \in 2$ put
	$Q_i = \bigcap \set{ T_{k_n}(k_{n+1}) \mid k_{n+1} \in \mathcal A^{g_i} }$.
	The sets
	$d_i = \bigcup \set{ [k_n, k_{n+1}) \mid k_{n+1} \in \mathcal A^{g_i} }$
	are disjoint for $i \in 2$,
	therefore for at least one $i \in 2$ is
	$\overline{d_i} = \bigcup \set{ e_k^T \mid k \in d_i } \in \mathcal I$.
	For this $i$ is $Q_i \in \mathbf Q_{\mathcal I}$.
	To check this, notice that
	for $k \in \omega$, $e^{Q_i}_k \in E^{Q_i}$
	there is some $e' \in E^T$ such that 
	$e^{Q_i}_k \subseteq e' \cup \overline{d_i} \cup (\omega \setminus \dom E^T)$.
	Moreover $\dom E^T \subset \dom E^{Q_i} \cup \overline{d_i}$,
	thus $\dom E^{Q_i} \in \mathcal{I}^*$.
	Since $Q_i < T_{k_n}(k_{n+1})$ for each $k_{n+1} \in \mathcal A^{g_i}$,
	and all but finitely many elements of $\mathcal A^{g_i}$ are of the form $k_{n+1}$,
	we have that $Q_i \Vdash \mathcal A^{g_i} \subset^* X$.
\end{proof}

Let $\mathcal A$ be a dense independent system and let $B$ be a free sequence.
We say that $B$ is a \emph{free sequence associated} with $\mathcal A$ if
$B$ is a maximal free sequence and $B$ generates the filter $\mathscr F_{\mathcal A}$.

\begin{theorem}\label{thm:FSpreserv}
	Let $B$ be a maximal free sequence associated with
	a dense independent system $\mathcal A$ in a model of set theory $V$.
	Let $W$ be a Cohen-preserving extension of $V$ such that $\mathcal A$
	remains dense in $W$.
	Then $B$ is a maximal free sequence associated with
	$\mathcal A$ in $W$.
\end{theorem}
\begin{proof}
	Lemma~\ref{lem:F_pres} states that $\mathscr F_{\mathcal A} \cap V$
	generates $\mathscr F_{\mathcal A}$ in $W$ so it remains to show that
	$B$ is a maximal free sequence in $W$.
	Take $X \subset \omega$ in $W$, we need to show that $B$ cannot be
	end-extended by $X$.
	Let $\mathscr F_{\mathcal A}^*$ be the ideal dual to $\mathscr F_{\mathcal A}$. 
	If $X \in \mathscr F_{\mathcal A}^*$, we are done so suppose this is not the case.
	Since $\mathcal A$ is dense in $W$ 
	we have that $X \notin \mathscr F_{\mathcal A}^*$ iff 
	there exists $h \in \mathbf C_{\mathcal A}$
	such that $\mathcal A^h \subset^* X$ ($h$ hits $X$). 
	As $\mathcal A^h \in V$, $\mathcal A^h \notin \mathscr F_{\mathcal A}^*$,
	and $B$ cannot be end-extended by $\mathcal A^h$,
	there is $b \in \comb{B}$ such that $b \subset^* \mathcal A^h$.
	Now $X^0 \cap b =^* \emptyset$ witnesses that $B$ cannot be
	end-extended by $X$.
\end{proof}


\begin{proposition}
	Assume $\mathfrak t = \mathfrak c$
	and let $\mathcal T$ be a tower. 
	There is a maximal decreasing free sequence
	$ \set{a_\alpha \mid \alpha \in \mathfrak c}$
	which is cofinal with $\mathcal T$.
\end{proposition}
\begin{proof}
	Let $\mathcal F$ be the filter generated by $\mathcal T$.
	If $\mathcal F$ is an ultrafilter, we are done.
	If this is not the case, fix an enumeration
	$\set{X_\alpha \mid \alpha \in \mathfrak c, \alpha \text{ even}}$
	of  $\pw{\omega} \setminus(\mathcal{F}\cup\mathcal{F}^*)$.
	We construct the tower $ \set{a_\alpha \mid \alpha \in \mathfrak c}$ cofinal in $\mathcal T$
	by induction.
	If $\beta < \mathfrak c$ is even and $a_\alpha$ is defined for each $\alpha < \beta$,
	find $t \in \mathcal T$ such that $t \subset^* a_\alpha$ and $a_\alpha \setminus t$
	is infinite for each $\alpha < \beta$,
	and let $a_\beta = t$ (choose $a_0 \in \mathcal T$ arbitrary).
	Then find $s \in \mathcal T$ such that
	$(t \setminus s) \cap X_\beta$ is infinite
	(use the assumptions on $\mathcal T$ and $X_\beta$)
	and let $a_{\beta + 1} = s \cup (t \setminus X_\beta)$.
	Notice that $a_\beta \setminus a_{\beta + 1}$ is an infinite subset of $X_\beta$.
	Now it is easy to check that the sequence we defined is a maximal free sequence.
\end{proof}

\begin{corollary}\label{cor:FSexists}
	Assume \textsf{CH}.
	For every selective independent system $\mathcal{A}$
	there exists a free sequence $B$ associated with $\mathcal{A}$.
\end{corollary}

\begin{theorem}
	It is consistent that
	$\omega_1 = \mathfrak i = \mathfrak f < \mathfrak u = \mathfrak c = \omega_2$.
\end{theorem}
\begin{proof}
	Start in a model of \textsf{CH} and run a countable support
	iteration of length $\omega_2$ of posets of form $\mathbf Q_{\mathcal I}$
	with the parameter $\mathcal I$ ranging over all maximal ideals on $\omega$ in
	all intermediate models.
	Lemma~\ref{lem:kill} together with the usual reflection argument
	implies that the final generic extension does not contain any
	ultrafilter base of size $\omega_1$, i.e.\ $\mathfrak u = \mathfrak c = \omega_2$.

	Use Proposition~\ref{prop:sel_sys} to find a selective
	independent system in the ground-model. Theorem~\ref{thm:onestep_preserv}
	states that the independent system remains selective in all successor
	stages of the iteration and Theorem~\ref{thm:limit_preserv} together with
	Remark~\ref{rem:preserve-dense}
	ensure that it remains selective
	also in limit stages of the iteration. Thus the ground-model
	independent system remains selective and in particular maximal in the final extension,
	$\mathfrak i = \omega_1$.
	Finally use Corollary~\ref{cor:FSexists} in the ground-model to
	find a free sequence associated with a selective independent system.
	Theorem~\ref{thm:FSpreserv} states that this free sequence is still
	maximal in the final generic extension, thus $\mathfrak f = \omega_1$.
\end{proof}

It is worth noting that in the resulting model all the usually considered 
cardinal characteristics of the continuum, except $\mathfrak u$, 
are equal to $\omega_1$. 
For $\mathfrak a$ this was proved by Guzm\'{a}n~\cite{osvaldo}.

\appendix

\section*{Appendix: Preservation theorem for the iteration}\label{sec:preservation}


The forcing iteration argument in Section~\ref{sec:mis}
uses a typical preservation theorem for countable support forcing iteration,
in this instance the preservation of a filter--co-filter pair.  
This theorem follows the usual pattern described
in~\cite{shelah-proper,goldstern-tools}.
However, as specific instances of preservation theorems are sometimes
difficult to derive from the general statements given in these sources,
we decided to provide the proof of the relevant preservation theorem in this appendix,
making the paper more self-contained.

Let $\mathcal F$ be a filter on $\omega$. We will use the following game $\mathrm G({\mathcal F})$.
Players~I and~II alternate for $\omega$ many rounds.
In the $n$-th round player~I plays a set $F_n \in \mathcal F$,
and player~II responds with $a_n \in F_n$. Player~II wins if
$\set{ a_n \mid n \in \omega } \in \mathcal F$.
The following is well known.

\begin{fact}\label{fact:game}
	Player~I does not have a winning strategy in the game $\mathrm G({\mathcal F})$ iff
	$\mathcal F$ is a rare P-filter.
\end{fact}

\begin{theorem}\label{thm:limit_preserv}
	Let $\mathcal F$ be a P-filter on $\omega$,
	denote $\mathcal K = \pw{\omega} \setminus \mathcal F$.
	For $\delta$ limit let $P_\delta = \la P_\alpha, Q_\alpha \mid \alpha < \delta \ra$
	be a countable support iteration of proper
	forcing notions such that for each $\alpha < \delta$
	\[P_\alpha \Vdash \mathcal F \text { is a rare filter and }
		\la \mathcal F \ra \cup \la \mathcal K \ra = \pw{\omega}.\]
	Then also $P_\delta \Vdash \la \mathcal F \ra \cup \la \mathcal K \ra = \pw{\omega}$.
\end{theorem}

By $\la \mathcal F \ra$ and $\la \mathcal K \ra$ we denote the upwards,
respectively downwards closure of
$\mathcal F$ and $\mathcal K$ in the appropriate models.
The assumption for $\alpha = 0$ states that $\mathcal F$ is a rare P-filter
in the ground model $V$.
Standard arguments shows that $\la \mathcal F \ra$ is a P-filter
in any generic extension via a proper forcing,
and $\la \mathcal F \ra$ is rare in any generic extension via a bounding forcing.

\begin{proof}
	If the cofinality of $\delta$ is uncountable, no new reals are added
	at stage $\delta$ of the iteration,
	and the conclusion of the theorem holds true.
	Therefore we will assume that the cofinality of $\delta$ is countable,
	and by passing to a cofinal sequence of $\delta$,
	it is sufficient to prove the theorem in case $\delta = \omega$.
	In the following $G_\alpha$ denotes exclusively generic filters on $P_\alpha$.
	We use  $P$ to denote posets $P_\delta / G_\alpha$ in the intermediate generic
	extensions $V[G_\alpha]$. 
	Let $X$ be a $P$-name for a subset of $\omega$.
	For $r \in P$ let $X_r = \set{ n \in \omega \mid r \not\Vdash n \notin X }$.

	\begin{lemma}\label{lem:use-rare}
		Let $\mathcal H$ be a rare P-filter and $p \in P$ a condition.
		If $X_r \in \mathcal H$ for each $r < p$,
		then there exists $H \in \mathcal H$ and a sequence
		$\la r_i \in P \mid i \in \omega \ra$, $r_0 = p$, $r_{i+1} < r_i$
		such that $r_i \Vdash H \cap i \subset X$ for each $i\in \omega$.
	\end{lemma}
	\begin{proof}
		Put $p_0 = p$ and play the game $\mathrm G({\mathcal H})$
		as follows. In the $n$-th round player~I plays the set $X_{p_n} \in \mathcal H$,
		player~II responds with some $a_n \in X_{p_n}$. Player~I then chooses
		$p_{n+1} \in P$, $p_{n+1} < r_n$ such that $p_{n+1} \Vdash a_n \in X$
		and proceeds to the next round.
		Since $\mathcal H$ is a rare P-filter, this strategy is not winning for
		player~I. Thus there is a sequence of moves of player~II and conditions
		$\la p_n \mid n \in \omega \ra$ such that player~II wins the game,
		i.e.\ $H = \set{ a_n \mid n \in \omega } \in \mathcal H$.
		A sequence of conditions $\la r_i \mid i \in \omega \ra$ such that
		$r_i = p_{a_n}$ if $a_n < i \leq a_{n+1}$
		is as required in the lemma.
	\end{proof}

	Let $p$ be a condition in $P_\omega$. The goal is to find a stronger condition which forces
	either $X \in \la \mathcal F \ra$ or $X \in \la \mathcal K \ra$.
	In case there exists an intermediate extension $V[G_\alpha]$, $p \in G_\alpha$
	and $r \in P/G_\alpha$, $r < p/G_\alpha$
	such that $X_r \notin \la \mathcal F \ra$ (in $V[G_\alpha]$),
	then $r \Vdash X \in \la \mathcal K \ra$ due to the assumption of the theorem,
	and there exists a condition in $P_\omega$ stronger than $p$ forcing the same statement.
	Therefore we will assume in the rest of the proof that this is not the case.

	For a sufficiently large $\theta$ fix a countable elementary submodel $N \prec H(\theta)$
	such that $X, p, \mathcal F, P_\omega \in N$.
	Use Lemma~\ref{lem:use-rare} in $N$ for $\mathcal H = \mathcal F$ and $P = P_\omega$
	to get $H \in \mathcal F \cap N$ and
	a sequence $\la r^0_n \in P_\omega \mid n \in \omega \ra \in N$.
	Since $\mathcal F$ is a P-filter, there exists $A^* \in \mathcal F$ such that $A^* \subset H$, and
	$A^* \subset^* F$ for each $F \in \mathcal F \cap N$.

	\begin{lemma}\label{lem:use-master}
		Let $q$ be a $(P_i,N)$-master condition, and let $\la F_n \mid n \in \omega\ra \in N[G_i]$
		be a sequence of elements of $\mathcal F$.
		Then \[q \Vdash \text{ There are infinitely many } n \in \omega
			\text{ such that } A^* \setminus n \subset F_n.\]
	\end{lemma}
	\begin{proof}
		Since $N[G_i] \prec H(\theta)[G_i]$ and $\mathcal F$ generates
		a non-meager filter in $H(\theta)[G_i]$,
		there is $F \in \mathcal F \cap N[G_i]$ such that $F \setminus n \subset F^n$
		for infinitely many~$n$
		(Fact~\ref{fact:non-meager}). Now $q \Vdash F \in N$ and we can use that $A^* \subset^* F$.
	\end{proof}

	We will inductively construct a condition $q < p$ such that $q \Vdash A^* \subset X$.
	Specifically, we construct two sequences of conditions $p_i, q_i$ for $i \in \omega$
	with the following properties;

	\begin{enumerate}
		\item\label{item:standard}
		      \begin{itemize}
			      \item $p_i \in P_\omega$,
			      \item $p_{i+1} < p_i$,
			      \item $p_{i+1} \restriction i = p_{i} \restriction i$,
			      \item $q_i \in P_i$,
			      \item $q_{i+1} \restriction i = q_{i}$,
			      \item $q_i < p_{i} \restriction i$,
			      \item $q_i$ is a $(N, P_i)$-master condition;
		      \end{itemize}
		\item\label{item:goal} $q_i \Vdash (p_i/G_i \Vdash A^* \cap i \subset X)$,
		\item\label{item:inductive} $q_i \Vdash \big(\text{There is a sequence }
			      \la r^i_n \in P_\omega/G_i \mid n \in \omega \ra \in N[G_i],\\
			      r^i_n < p_i/G_i \text{ such that } r^i_n \Vdash A^* \cap n \subset X\big)$.
	\end{enumerate}

	The construction starts with putting $p_0 = p$ and
	let $q_0$ be a trivial condition (in the trivial forcing $P_0$).
	Existence of the sequence $\la r^0_n \in P_\omega \mid n \in \omega \ra$
	follows from the choice of $A^*$.

	Suppose that $p_i, q_i$ are defined, work in $N[G_i]$ assuming $q_i \in G_i$.
	For each $n \in \omega$ consider a model $N[G_{i+1}]$
	such that $r^i_n \restriction (i+1) \in G_{i+1}/G_i$.
	Use Lemma~\ref{lem:use-rare} in $N[G_{i+1}]$ for $\la \mathcal F \ra$ and $r^i_n/G_{i+1}$
	to get $H_n \in \la \mathcal F \ra \cap N[G_{i+1}]$ and a sequence
	$\la s^n_k \in P_\omega/G_{i+1} \mid k \in \omega \ra \in N[G_{i+1}]$ as in the lemma.
	We can assume that $H_n \in \mathcal F \cap N[G_{i+1}]$, and by strengthening
	$r^i_n \restriction \set{ i }$ to $t^i_n \restriction \set{ i } \in N[G_i]$ we can decide
	$H_n$ to be some $F_n \in \mathcal F \cap N[G_i]$.
	Since $q_i$ is $(N, P_i)$-master, Lemma~\ref{lem:use-master} implies that there is $m > i$
	such that $A^* \setminus m \subset F_m$.

	Define $p_{i+1} = p_i \restriction i \conc t^i_m$, and let
	$q_{i+1} < p_{i+1}\restriction i+1 $ be any $(N, P_{i+1})$-master condition
	such that $q_{i+1} \restriction i = q_i$.
	Property~(\ref{item:standard}) is obviously satisfied.
	Property (\ref{item:goal}) follows from $m > i$, the inductive hypothesis for $r_m^i$,
	and from ${{q_{i+1}} \conc (p_{i+1}/G_{i+1})} < q_i \conc r_m^i$.
	To justify~(\ref{item:inductive}) notice that $q_{i+1}$ forces that the sequence
	$\la s^m_k \mid k \in \omega \ra$ satisfies the condition required for
	$\la r^i_n \mid n \in \omega \ra$;
	for $y \in A^* \cap m$ this follows from the inductive hypothesis on $r^i_m$,
	and for $y \in A^*, x \geq m$ from the choice of $\la s^m_k \mid k \in \omega \ra$
	and $A^* \setminus m \subset F_m$.

	Once the inductive construction is done, the condition
	$q = \bigcup \set{ q_i \mid i \in \omega }$ forces that
	$A^* \subset X$. The inclusion $A^* \cap i \subset X$
	is guaranteed by property~(\ref{item:goal}) and $q < q_i \conc {(p_i/G_i)}$.
\end{proof}

\section*{Acknowledgments}

\noindent
The authors would like to thank Osvaldo Guzm\'{a}n
for numerous
suggestions substantially improving the paper.

\bibliography{reference}
\bibliographystyle{amsalpha}

\end{document}